\documentclass[12pt]{article}
\usepackage{graphicx}
\usepackage{amscd, amssymb}
\usepackage{enumerate}
\usepackage{hyperref}
\usepackage{lscape}
\newenvironment{eq}{\begin{equation}}{\end{equation}}
\newenvironment{proof}{{\bf Proof}:}{\vskip 5mm }

\newtheorem{proposition}{Proposition}[subsection]
\newtheorem{lemma}[proposition]{Lemma}
\newtheorem{definition}[proposition]{Definition}
\newtheorem{theorem}[proposition]{Theorem}

\newtheorem{example}[proposition]{Example}
\newtheorem{remark}[proposition]{Remark}

\newtheorem{problem}[proposition]{Problem}
\newtheorem{construction}[proposition]{Construction}
\newcommand{\llabel}[1]{\label{#1}}
\newcommand{\comment}[1]{}
\newcommand{\sr}{\rightarrow}

\newcommand{\uu}{\underline}

\newcommand{\wt}{\widetilde}

\newcommand{\dd}{\diamond}

{\vskip
3mm}

\newcommand{\spc}{{\,\,\,\,\,\,\,}}

\begin{document}
\parskip = 2mm
\begin{center}
{\bf\Large Products of families of types in the C-systems defined by a universe category}


\vspace{3mm}

{\large\bf Vladimir Voevodsky}\footnote{School of Mathematics, Institute for Advanced Study,
Princeton NJ, USA. e-mail: vladimir@ias.edu}$^,$\footnote{Work on this paper was supported by NSF grant 1100938.}
\vspace {3mm}

{July 2015}  
\end{center}
\begin{abstract}
We introduce the notion of a $(\Pi,\lambda)$-structure on a C-system and show that C-systems with $(\Pi,\lambda)$-structures are constructively equivalent to contextual categories with products of families of types.  We then show how to construct $(\Pi,\lambda)$-structures on C-systems of the form $CC({\cal C},p)$ defined by a universe $p$ in a locally cartesian closed category $\cal C$ from a simple pull-back square based on $p$.  In the last section we prove a theorem that asserts that our construction is functorial. 
\end{abstract}

\subsection{Introduction}

The concept of a C-system in its present form was introduced in \cite{Csubsystems}. The type of the C-systems is constructively equivalent to the type of contextual categories defined by Cartmell in \cite{Cartmell1} and \cite{Cartmell0} but the definition of a C-system is slightly different from the Cartmell's foundational definition.

In this paper we consider what might be the most important structure on C-systems - the structure that corresponds, for the syntactic C-systems, to the operations of dependent product,  $\lambda$-abstraction and application. A C-system formulation of this structure was introduced by John Cartmell in \cite[pp. 3.37 and 3.41]{Cartmell0} as a part of what he called a strong M.L. structure. It was studied further by Thomas Streicher in \cite[p.71]{Streicher} who called a C-system (contextual category) together with such a structure a ``contextual category with products of families of types''. 

We first show that the structure that Cartmell defined is equivalent to another structure, which we call a $(\Pi,\lambda)$-structure. The proof of this equivalence consists of Constructions \ref{2015.03.13.constr1} and \ref{2015.03.15.constr1} (of mappings in both directions) and Lemmas \ref{2015.03.15.l1} and \ref{2015.03.15.l2} showing that these mappings are mutually inverse. 

Then we consider the case of C-systems of the form $CC({\cal C},p)$ introduced in \cite{Cfromauniverse}. They are defined, in a functorial way, by a category $\cal C$ with a final object and a morphism $p:\wt{U}\sr U$ in $\cal C$  together with the choice of pull-backs of $p$ along all morphisms in $\cal C$. A morphism with such choices is called a universe in $\cal C$. 
An important feature of this construction is that the C-systems $CC({\cal C},p)$ corresponding to different choices of pull-backs and different choices of final objects are canonically isomorphic. This fact makes it possible to say that $CC({\cal C},p)$ is defined by $\cal C$ and $p$. 

We provide several intermediate results about $CC({\cal C},p)$ when $\cal C$ is a locally cartesian closed category leading to the main result of this paper - Construction \ref{2015.03.17.constr3} that produces a $(\Pi,\lambda)$-structure on $CC({\cal C},p)$ from a simple pull-back square based on $p$. This construction was first announced in \cite{CMUtalk}.  It and the ideas that it is based on are among the most important ingredients of the construction of the univalent model of the Martin-Lof type theory. 

The methods of this paper are fully constructive. It is also written in the formalization-ready style that is in such a way that no long arguments are hidden even when they are required only to substantiate an assertion that may feel obvious to readers who are closely associated with a particular tradition of mathematical thought. 

In this paper we continue to use the diagrammatic order of writing composition of morphisms, i.e., for $f:X\sr Y$ and $g:Y\sr Z$ the composition of $f$ and $g$ is denoted by $f\circ g$.

I am grateful to the Department of Computer Science and Engineering of the University of Gothenburg and Chalmers University of Technology for its the hospitality during my work on the paper.


\subsection{Products of families of types and $(\Pi,\lambda)$-structures}
Let $CC$ be a C-system.  For $\Gamma\in Ob(CC)$, let $Ob_{n}(\Gamma)$ be the set of elements $\Delta$ in $Ob$ such that $l(\Delta)\ge n+l(\Gamma)$ and $ft^n(\Delta)=\Gamma$ and $\wt{Ob}_n(\Gamma)$ the set of elements $s\in\wt{Ob}(CC)$ such that $s:ft(\Delta)\sr \Delta$ where $\Delta\in Ob_n(\Gamma)$. For $n=0$ we will abbreviate $\wt{Ob}_0(\Gamma)$ as $\wt{Ob}(\Gamma)$. Note that in view of the definition of $\wt{Ob}$ we have $\wt{Ob}(X)=\emptyset$ if $l(X)=0$. 

For $f:\Gamma'\sr \Gamma$ the functions $\Delta \mapsto f^*(\Delta,n)$ and $s\mapsto f^*(s,n)$, defined in \cite{Csubsystems} as iterated canonical pull-backs of objects and sections respectively, give us functions:
$$Ob_n(\Gamma)\sr Ob_n(\Gamma')$$
$$\wt{Ob}_n(\Gamma)\sr \wt{Ob}_n(\Gamma')$$
which we will write simply as $f^*$. 

Let us note also that if $\Delta,\Delta'\in Ob(\Gamma)$, 
$u:\Delta\sr \Delta'$ is a morphism over $\Gamma$ and $f:\Gamma'\sr \Gamma$ is a morphism then, using the fact the the canonical squares are pull-back, we get a morphism $f^*(\Delta)\sr f^*(\Delta')$ that we denote by $f^*(u)$. 

The structure of ``products of families of types'' is defined in \cite[pp.3.37 and 3.41]{Cartmell0} and also considered in \cite[p.71]{Streicher}. Let us remind this definition here. 
\begin{definition}
\llabel{2015.03.17.def1}
The structure of products of families of types on a C-system $CC$ is a collection of data of the form:
\begin{enumerate}
\item for every $\Gamma\in Ob$ a function $\Pi_{\Gamma}:Ob_2(\Gamma) \sr Ob_1(\Gamma)$, which we write simply as $\Pi$,
\item for every $\Gamma$ and $B\in Ob_2(\Gamma)$ a morphism $Ap_B:p_{A}^*(\Pi(B))\sr B$ over $A$, where $A=ft(B)$,
\end{enumerate}
such that:
\begin{enumerate}
\item  for any $\Gamma$ and $B\in Ob_2(\Gamma)$ the map $\lambda inv_{Ap}:\wt{Ob}(\Pi(B))\sr \wt{Ob}(B)$ defined as:
$$s\mapsto p_A^*(s)\circ Ap_B$$
is a bijection,
\item for any $f:\Gamma'\sr \Gamma$ the square 
$$
\begin{CD}
Ob_2(\Gamma) @>\Pi_{\Gamma}>> Ob_1(\Gamma)\\
@Vf^*VV @VVf^*V\\
Ob_2(\Gamma') @>\Pi_{\Gamma'}>> Ob_1(\Gamma')\\
\end{CD}
$$
commutes,
\item for any $\Gamma$, $B\in Ob_2(\Gamma)$ and $f:\Gamma\sr \Gamma'$ one has $f^*(Ap_B)=Ap_{f^*(B)}$.
\end{enumerate} 
\end{definition}
We will show in the next section how to construct products of families of types on C-systems of the form $CC({\cal C},p)$. For this construction we first need to introduce another structure on C-systems and construct a bijection between the set of the products of families of types structures and this new structures.
\begin{definition}
\llabel{2015.03.09.def1}
Let $CC$ be a C-system. A pre-$(\Pi,\lambda)$-structure on $CC$ is a pair of families of functions
$$\Pi_{\Gamma}:Ob_2(\Gamma)\sr Ob_1(\Gamma)$$
$$\lambda_{\Gamma}:\wt{Ob}_2(\Gamma)\sr \wt{Ob}_1(\Gamma)$$
such that $\partial(\lambda(s))=\Pi(\partial(s))$ and one has:
\begin{enumerate}
\item for any $f:\Gamma'\sr \Gamma$ the square
\begin{eq}
\llabel{2015.03.09.sq1}
\begin{CD}
Ob_2(\Gamma) @>\Pi_{\Gamma}>> Ob_1(\Gamma)\\
@Vf^*VV @VVf^*V\\
Ob_2(\Gamma') @>\Pi_{\Gamma'}>> Ob_1(\Gamma')\\
\end{CD}
\end{eq}
commutes,
\item for any $f:\Gamma'\sr \Gamma$ the square
\begin{eq}
\llabel{2015.03.09.sq2}
\begin{CD}
\wt{Ob}_2(\Gamma) @>\lambda_{\Gamma}>> \wt{Ob}_1(\Gamma)\\
@Vf^*VV @VVf^*V\\
\wt{Ob}_2(\Gamma') @>\lambda_{\Gamma'}>> \wt{Ob}_1(\Gamma')\\
\end{CD}
\end{eq}
commutes.
\end{enumerate}
\end{definition}
The condition that $\partial(\lambda(s))=\Pi(\partial(s))$ can also be seen as the assertion that the square:
\begin{eq}
\llabel{2015.03.09.eq1}
\begin{CD}
\wt{Ob}_2(\Gamma) @>\lambda_{\Gamma}>> \wt{Ob}_1(\Gamma)\\
@V\partial VV @VV \partial V\\
Ob_2(\Gamma) @>\Pi_{\Gamma}>> Ob_1(\Gamma)
\end{CD}
\end{eq}
commutes.  
\begin{definition}
\llabel{2015.03.09.def2}
A pre-$(\Pi,\lambda)$-structure is called a $(\Pi,\lambda)$-structure if for any $\Gamma\in Ob_{\ge 2}$ the square (\ref{2015.03.09.eq1}) is a pull-back square or, equivalently, if the functions 
$$\lambda'_{\Gamma}:\wt{Ob}(\Gamma)\sr \wt{Ob}(\Pi(\Gamma))$$
defined by $\lambda_{\Gamma}$ are bijections. 
\end{definition}

We are going to show that, for a given family of functions $\Pi_{\Gamma}$, the type of $(\Pi,\lambda)$-structures over $\Pi_{\Gamma}$ is equivalent to the type of products of families of types over the same $\Pi_{\Gamma}$.

We first reformulate the structure of products of families slightly. Instead of considering $p_A^*(\Pi(B))$ we will consider an object that is isomorphic (but not equal!) to it, namely $p_{\Pi(B)}^*(A)$. Our structure will then be a family of maps $\Pi$ as before together with, for every $\Gamma$ and $B\in Ob_2(\Gamma)$, a morphism $Ap'_B:p_{\Pi(B)}^*(A)\sr B$ over $A$ such that the map $\lambda inv'_{Ap'}:\wt{Ob}(\Pi(B))\sr \wt{Ob}(B)$ defined as:
$$s\mapsto q(s,p_{\Pi(B)}^*(A))\circ Ap'_B$$
is a bijection. This can be seen on the following diagram that also contains other elements that will be needed in the construction below. 
\begin{eq}
\llabel{2015.03.13.eq1}
\begin{CD}
B @>q(s,p_{\Pi(B)}^*(B,2),2)>> p_{\Pi(B)}^*(B,2) @>q(p_{\Pi(B)},B,2)>> B\\
@Vp_B VV @VVV @VVp_B V\\
A @>q(s,p_{\Pi(B)}^*(A))>> p_{\Pi(B)}^*(A) @>q(p_{\Pi(B)},A)>> A\\
@Vp_A VV @VVV @Vp_A VV\\
\Gamma @>s>> \Pi(B) @>p_{\Pi(B)}>> \Gamma
\end{CD}
\end{eq}

We now state the problem which we will provide a construction for:
\begin{problem}
\llabel{2015.03.13.prob1}
Let $CC$ be a C-system and let $\Pi$ be a family of functions 
$$\Pi_{\Gamma}:Ob_2(\Gamma)\sr Ob_1(\Gamma)$$ 
given for all $\Gamma\in Ob$ such that the corresponding squares of the form (\ref{2015.03.09.sq1}) commute. 

To construct a bijection between the following two types of structure:
\begin{enumerate}
\item for every $\Gamma$ and $B\in Ob_2(\Gamma)$ a bijection  
$$\lambda'_{B}:\wt{Ob}(B)\sr \wt{Ob}(\Pi(B))$$
such that for every morphism $f:\Gamma'\sr \Gamma$ the square 
$$
\begin{CD}
\wt{Ob}(B) @>\lambda'_B >> \wt{Ob}(\Pi(B))\\
@Vf^* VV @VV f^* V\\
\wt{Ob}(f^*(B)) @>\lambda'_{f^*(B)} >> \wt{Ob}(\Pi(f^*(B)))
\end{CD}
$$
defined by $f$, commutes. 
\item  for every $\Gamma\in Ob$ and $B\in Ob_2(\Gamma)$ a morphism $Ap'_{B}:p_{\Pi(B)}^*(A)\sr B$ over $A$, where $A=ft(B)$, such that the map 
$$\lambda inv'_{Ap'}:\wt{Ob}(\Pi(B))\sr \wt{Ob}(B)$$
defined as:
$$s\mapsto q(s,p^*_{\Pi(B)}(A))\circ Ap'_B$$
is a bijection and such that for every morphism $f:\Gamma'\sr \Gamma$ and $B\in Ob_2(\Gamma)$ one has $f^*(Ap'_B)=Ap'_{f^*(B)}$. 
\end{enumerate}
\end{problem}
We will construct the solution in four steps - first a function from structures of the first kind to structures of the second, then a function in the opposite direction and the  two lemmas proving that the first function is a left and a right inverse to the second. 

\begin{construction}
\llabel{2015.03.13.constr1}\rm
Let us show how to construct a structure of the second kind from a structure of the first kind. 
To define $Ap'$ consider the digram of $\Pi$'s  defined by the diagram (\ref{2015.03.13.eq1}):
\begin{eq}
\llabel{2015.03.13.eq2}
\begin{CD}
\Pi(B) @. \Pi(p_{\Pi(B)}^*(B,2)) @. \Pi(B)\\
@VVV @VVV @VVV\\
\Gamma @>s>> \Pi(B) @>p_{\Pi(B)}>> \Gamma
\end{CD}
\end{eq}
Note that since $\Pi$ is stable under pull-backs we have 
$$\Pi(p_{\Pi(B)}^*(B,2))=p_{\Pi(B)}^*(\Pi(B))$$
and therefore the diagonal $\delta_{\Pi(B)}$ gives us an element in $\wt{Ob}(\Pi(p_{\Pi(B)}^*(B,2)))$. Applying to it the inverse of our $\lambda'$ we get an element $ap:\wt{Ob}(p_{\Pi(B)}^*(B,2))$. Define:
$$Ap'_B=ap\circ  q(p_{\Pi(B)},B,2)$$
Let us prove that these morphisms satisfy the conditions of bijectivity and the stability under pull-backs. We need to show that the mappings $\lambda inv'_{Ap'}:\wt{Ob}(\Pi(B))\sr \wt{Ob}(B)$ defined as:
$$s\mapsto q(s,p_{\Pi(B)}^*(A))\circ Ap'_B$$
are bijective. It is sufficient to show that the mappings $\lambda inv'_{Ap'}$ are inverse to the ones given by $\lambda'$ from at least one side as any inverse to a bijection is a bijection. 

We do it in two steps. First let 
$$\lambda inv''(s)=s^*(ap,2)=q(s,p_{\Pi(B)}^*(A))^*(ap)$$
Let us show that $\lambda inv'' = \lambda inv'_{Ap'}$. Indeed:
$$q(s,p_{\Pi(B)}^*(A))^*(ap)=q(s,p_{\Pi(B)}^*(A))^*(ap)\circ q(s,p_{\Pi(B)}^*(B,2),2)\circ q(p_{\Pi(B)},B,2)=$$
$$q(s,p_{\Pi(B)}^*(A))\circ ap\circ q(p_{\Pi(B)},B,2) = q(s,p_{\Pi(B)}^*(A))\circ Ap'_B$$
Now we have:
$$\lambda'(\lambda inv''(s))=\lambda'(s^*(ap,2))=s^*(\lambda'(ap),1)=s^*(\delta_{\Pi(B)},1)=s.$$
It remains to check that the mappings $Ap'$ are stable under the base change. Since the base change of morphisms commutes with compositions this follows if we know that $ap$ is stable and $q(-,-,2)$ is stable. The second fact is verified easily from the axioms of a C-system
 and the first follows from the stability of $\delta$ and the pull-back and the assumption that $\lambda'$ is stable under pull-back.  
\end{construction}
\begin{construction}
\llabel{2015.03.15.constr1}\rm
Let us now construct a structure of the first kind from a structure of the second. This is straightforward since a construction of the second kind gives is bijections $\lambda inv'_{Ap'}$ and the inverse to these bijections are bijections required for the structure of the first kind. The fact that the bijections that we obtain in this way are stable under the pull-backs follows from the fact that the pull-backs commute with compositions, that they take morphisms of the form $q(-,-,1)$ to morphisms of the same form and from our assumption that morphisms $Ap'$ are stable under composition. 
\end{construction}
Let us denote the map of Construction \ref{2015.03.13.constr1} by $C1$ and the map of Construction \ref{2015.03.15.constr1} by $C2$.
\begin{lemma}
\llabel{2015.03.15.l1}
For a structure of the first kind $\lambda'$ one has $C2(C1(\lambda'))=\lambda'$. 
\end{lemma}
\begin{proof}
This is immediate since in Construction \ref{2015.03.13.constr1} we proved that the $\lambda inv'_{Ap'}$ that we have constructed are bijections by showing that they are inverses to the $\lambda'$'s that we started with and in Construction \ref{2015.03.15.constr1} we defined $\lambda'$'s as inverses to $\lambda inv'_{Ap'}$.
\end{proof}
\begin{lemma}
\llabel{2015.03.15.l2}
For a structure of the second kind $Ap'$ one has $C1(C2(Ap'))=Ap'$. 
\end{lemma}
\begin{proof}
This amounts to checking that
$$\lambda inv'_{Ap'}(\Delta_{\Pi(B)})\circ q(p_{\Pi(B)},B,2)=Ap'_B$$
Opening up the definition of $\lambda inv'$ we get the equation
$$q(\delta_{\Pi(B)},p^*_{p^*_{\Pi(B)}(\Pi(B))}( p^*_{\Pi(B)}(A))) \circ Ap'_{p^*_{\Pi(B)}(B,2)}q(p_{\Pi(B)},B,2)=Ap'_B$$
We have for any $f:\Gamma'\sr \Gamma$:
$$Ap'_{f*(B,2)}\circ q(f,B,2)=q(q(f,\Pi(B)),p^*_{\Pi(B)}(A))\circ Ap'_B$$
and our equation becomes 
$$q(\delta_{\Pi(B)},p^*_{p^*_{\Pi(B)}(\Pi(B))}( p^*_{\Pi(B)}(A))) \circ q(q(p_{\Pi(B)},\Pi(B)),p^*_{\Pi(B)}(A))\circ Ap'_B=Ap'_B$$
Which follows from:
$$q(\delta_{\Pi(B)},p^*_{p^*_{\Pi(B)}(\Pi(B))}( p^*_{\Pi(B)}(A))) \circ q(q(p_{\Pi(B)},\Pi(B)),p^*_{\Pi(B)}(A))=$$
$$q(\delta_{\Pi(B)}\circ q(p_{\Pi(B)},\Pi(B)), p^*_{\Pi(B)}(A))=q(Id,p^*_{\Pi(B)}(A))=Id.$$
\end{proof}
This completes our construction for Problem \ref{2015.03.13.prob1}.

\subsection{More on the C-systems of the form $CC({\cal C},p)$}

Let us start by considering a general (pre-)category $\cal C$. Let $p:\wt{U}\sr U$ be a morphism in $\cal C$. Recall from \cite{Cfromauniverse} that a universe structure on $p$ is a choice of pull-back squares of the form
$$
\begin{CD}
(X;F) @>Q(F)>> \wt{U}\\
@Vp_{X,F} VV @VV p V\\
X @>F>> U
\end{CD}
$$
for all $X$ and all morphisms $F:X\sr U$. A universe in $\cal C$ is a morphism with a universe structure on it and a universe category is a category with a universe and a choice of a final object $pt$. 

We may use the notation $(X;F_1,\dots,F_n)$ for $(\dots (X;F_1);\dots F_n)$. 

For $f:W\sr X$ and $g:W\sr \wt{U}$ we will denote by $f*g$ the unique morphism such that 
$$(f*g)\circ p_{X,F}=f$$
$$(f*g)\circ Q(F)=g$$
For $X'\stackrel{f}{\sr}X\stackrel{F}{\sr}U$ we let $Q(f,F)$ denote the morphism 
$$(p_{X',f\circ F}\circ f)*Q(f\circ F):(X';f\circ F)\sr (X;F)$$

The construction of the C-system $CC({\cal C},p)$ presented in \cite{Cfromauniverse} can be described as follows. One defines first, by induction on $n$, pairs $(Ob_n, int_n:Ob_n\sr {\cal C})$ where $Ob_n=Ob_n({\cal C},p)$ is a set and $int_n$ is a function from $Ob_n$ to objects of $\cal C$ as follows:
\begin{enumerate}
\item $Ob_0$ is the standard one point set $unit$ whose element we denote by $tt$. The function $int_0$ maps $tt$ to the final object $pt$ of the universe category structure on $\cal C$,
\item $Ob_{n+1}=\amalg_{A\in Ob_n}Hom(int(A),U)$ and $int_{n+1}(A,F)=(int(A);F)$.
\end{enumerate}
We then define $Ob(CC({\cal C},p))$ as $\amalg_{n\ge 0}Ob_n$ such that elements of $Ob(CC({\cal C},p))$ are pairs $\Gamma=(n,A)$ where $A\in Ob_n({\cal C},p)$. We define the function $int:Ob(CC({\cal C},p))\sr {\cal C}$ as the sum of functions $int_n$. 

The morphisms in $CC({\cal C},p)$ are defined by
$$Mor_{CC({\cal C},p)}=\amalg_{\Gamma,\Gamma'\in Ob(CC)}Hom_{\cal C}(int(\Gamma),int(\Gamma'))$$
and the function $int$ on morphisms maps a triple $(\Gamma,(\Gamma',a))$ to $a$.
Note that the subset in $Mor$ that consists of $f$ such that $dom(f)=\Gamma$ and $codom(f)=\Gamma'$ is not equal to the set $Hom_{\cal C}(int(\Gamma),int(\Gamma'))$ but instead to the set of triples of the form $f=(\Gamma,(\Gamma',a))$ where $a\in Hom_{\cal C}(int(\Gamma),int(\Gamma'))$. 
\begin{problem}
\llabel{2015.04.30.prob1}
To construct, for all $\Gamma\in Ob(CC({\cal C},p))$ bijections
$$u_{1,\Gamma}:Ob_1(\Gamma)\sr Hom_{\cal C}(int(\Gamma),U)$$
$$\wt{u}_{1,\Gamma}:\wt{Ob}_1(\Gamma)\sr Hom_{\cal C}(int(\Gamma),\wt{U})$$
such that:
\begin{enumerate}
\item for $(n,A)\in Ob(CC({\cal C},p))$ one has 
\begin{eq}
\llabel{2015.04.30.eq3a}
u_1(n+1,(A,F))=F
\end{eq}
and if $l(\Gamma')=n>0$ then 
\begin{eq}
\llabel{2015.05.02.eq1a}
int(\Gamma')=(int(ft(\Gamma')); u_{1}(\Gamma'))
\end{eq}
\item for $o\in \wt{Ob}_1(\Gamma)$ one has 
\begin{eq}
\llabel{2015.04.30.eq4a}
\wt{u}_1(o)=int(o)\circ Q(u_1(\partial(o)))
\end{eq}
and
\begin{eq}
\llabel{2015.05.04.eq1a}
int(o)=Id_{ft(\partial(o))}*\wt{u}_1(o)
\end{eq}
\item $u_1$ and $\wt{u}_1$ are natural in $\Gamma$ i.e. for any $f:\Gamma'\sr \Gamma$ one has:
\begin{eq}
\llabel{2015.04.30.eq1a}
u_1(f^*(T))=f\circ u_1(T)
\end{eq}
\begin{eq}
\llabel{2015.04.30.eq2a}
\wt{u}_1(f^*(o))=f\circ \wt{u}_1(o)
\end{eq}
\item one has
\begin{eq}
\llabel{2015.05.02.eq5a}
u_1(\partial(o))=\wt{u}_1(o)\circ p
\end{eq}
\end{enumerate}
\end{problem}
\begin{remark}\rm
\llabel{2015.07.29.rem1}
The families of sets $Ob_1(\Gamma)$ and $\wt{Ob}_1(\Gamma)$ together with the families of functions $f^*$ satisfy the axioms of presheaves. To construct families of functions $u_{1,\Gamma}$ and $\wt{u}_{1,\Gamma}$ satisfying conditions (3) of the problem is the same as to construct presheaf isomorphisms $Ob_1\sr int_*(Yo(U))$ and $\wt{Ob}_1\sr int_*(Yo(\wt{U}))$ where $Yo$ is the Yoneda embedding and 
$$int_*:PreShv({\cal C})\sr PreShv(CC)$$
is the functor given by $int_*(F)(X)=F(int(X))$. The fourth condition asserts that the square
$$
\begin{CD}
\wt{Ob}_1 @>\wt{u}_1>> int_*(Yo(\wt{U}))\\
@V\partial VV @VV Yo(p)V\\
Ob_1 @>u_1>> int_*(Yo(U))
\end{CD}
$$
commutes.  
\end{remark}
\begin{construction}\rm
\llabel{2015.05.02.constr1}
For $\Gamma=(n,A)$ where $A\in Ob_n({\cal C},p)$, an element $\Gamma'$ in $Ob_1(\Gamma)$ is a triple $(n+1,(A,F))$ where $F:int(A)\sr U$. Mapping such a triple to $F$ we obtain a bijection
$$u_{1,\Gamma}:Ob_1(\Gamma)\sr Hom_{\cal C}(int(\Gamma),U)$$
For $\Gamma'$ such that $l(\Gamma')=n+1>0$ we have $\Gamma'=(n+1,(A,F))$ where $ft(\Gamma')=(n,A)$ and 
$$int(\Gamma')=int(A,F)=(int(A);F)=(int(ft(\Gamma)),u_1(\Gamma'))$$

An element $o$ in $\wt{Ob}_1(\Gamma)$ is a triple $(\Gamma,(\Gamma',s))$ where
$$s=int(o)\in Hom_{\cal C}(int(\Gamma),int(\Gamma')),$$
$\Gamma'=\partial(o)$ is an object such that $ft(\Gamma')=\Gamma$, $s\circ int(p_{\Gamma'})=Id_{int(\Gamma)}$ and $l(\Gamma')=n+1>0$. 

Define the function $\wt{u}_{1,\Gamma}$ by the formula
$$\wt{u}_{1,\Gamma}(o)=int(o)\circ Q(u_{1,\Gamma}(\partial(o)))$$
If $\Gamma=(n,A)$ then 
$$\partial(o)=(n+1,(A,F))$$
where $F=u_{1,\Gamma}(\partial(o)):int(A)\sr U$ and we have a canonical square
\begin{eq}\llabel{2015.05.04.eq2}
\begin{CD}
int(\partial(o)) @>Q(u_{1,\Gamma}(\partial(o)))>> \wt{U}\\
@Vint(p_{\partial(o)})VV @VVpV\\
int(\Gamma) @>u_{1,\Gamma}(\partial(o))>> U
\end{CD}
\end{eq}
which shows that the composition $s\circ Q(u_{1,\Gamma}(\Gamma'))$ is defined and is a morphism $int(\Gamma)\sr \wt{U}$. 

For the formula  (\ref{2015.05.04.eq1a}) we have
$$int(o)=Id_{ft(\partial(o))}*(int(o)\circ Q(u_1(\partial(o))))$$
because a morphism to a fiber product equals to the product of its composition with the projections and therefore
$$int(o)=Id_{ft(\partial(o))}*\wt{u}_1(o)$$
by definition of $\wt{u}_1(o)$. 

To show $\wt{u}_{1,\Gamma}$ it is a bijection let us construct an inverse. For $f:int(\Gamma)\sr \wt{U}$ let 
$$\wt{u}^!_{1,\Gamma}(f)=(\Gamma, ((n+1, (A,f\circ p)), s_f))$$
where $s_f:int(\Gamma)\sr (A;f\circ p)$ is the unique section of $p_{A,f\circ p}$ such that $s_f\circ Q(f\circ p)=f$. 

We have 
$$\wt{u}^!(\wt{u}(\Gamma,(\Gamma',s)))=\wt{u}^!(s\circ Q(u(\Gamma')))=$$$$(\Gamma,((n+1,(A,s\circ Q(u(\Gamma'))\circ p)),s'))=(\Gamma,((n+1,(A,u(\Gamma')),s')))$$
where $s'=s_{s\circ Q(u(\Gamma'))}=s$ which proves that $\wt{u}^!$ is inverse to $\wt{u}$ from one side. In the opposite direction we have
$$\wt{u}(\wt{u}^!(f))=\wt{u}(\Gamma,((n+1,(A,f\circ p)),s_f))=s_f\circ Q(u((n+1,(A,f\circ p))))=s_f\circ Q(f\circ p)=f$$

The proofs of the naturality of $u_1$ and $\wt{u}_1$ with respect to morphisms in $\Gamma$ follow easily from the definition of the canonical squares in $CC({\cal C},p)$. 

Formula (\ref{2015.05.02.eq5a}) is a corollary of the commutativity of the square (\ref{2015.05.04.eq2}). 
\end{construction}

We will now construct bijections $u_{2,\Gamma}$ and $\wt{u}_{2,\Gamma}$ similar to the bijections $u_{1,\Gamma}$ and $\wt{u}_{1,\Gamma}$ but having as sources $Ob_2(\Gamma)$ and $\wt{Ob}_2(\Gamma)$. 

For any $V\in{\cal C}$ we define functor data $D_p(-,V)$ given on objects by
$$D_p(X,V) := \amalg_{F:X\sr U}Hom((X;F),V)$$
and on morphisms by
$$D_p(f,V):(F_1,F_2)\mapsto (f\circ F_1, Q(f,F_1)\circ F_2)$$
The sets $D_p(X,V)$ are also functorial in $V$ according to the formula
$$D_p(X,g)(F_1,F_2)=(F_1,F_2\circ g)$$
and for $f:X\sr X'$, $g: V\sr V'$ we have
$$D_p(f,V)\circ D_p(X,g)=D_p(X',g)\circ D_p(f,V')$$
\begin{problem}
\llabel{2015.05.02.prob2}
To construct for all $\Gamma\in Ob(CC({\cal C},p))$ bijections 
$$u_{2,\Gamma}:Ob_2(\Gamma)\sr D_p(int(\Gamma),U)$$
$$\wt{u}_{2,\Gamma}:\wt{Ob}_2(\Gamma)\sr D_p(int(\Gamma),\wt{U})$$
such that:
\begin{enumerate}
\item $u_{2,\Gamma}(T)=(u_{1,\Gamma}(ft(T)),u_{1,ft(T)}(T))$
\item $\wt{u}_{2,\Gamma}(o)=(u_{1,\Gamma}(ft(\partial(o))),\wt{u}_{1,ft(\partial(o))}(o))$
\item for $f:\Gamma'\sr \Gamma$ one has 
$$u_2(f^*(T))=D_p(f,U)(u_2(T))$$
$$\wt{u}_2(f^*(o))=D_p(f,\wt{U})(\wt{u}_2(o))$$
\item $u_2(\partial(o))=D_p(int(\Gamma),p)(\wt{u}_2(o))$
\end{enumerate}
\end{problem}
\begin{construction}\rm
\llabel{2015.05.02.constr2}
By (\ref{2015.05.02.eq1a}) we have 
$$int(ft(T))=(int(\Gamma);u_{1,\Gamma}(ft(T)))$$
and therefore $(u_{1,\Gamma}(ft(T)),u_{1,ft(T)}(T))$ is a well defined element of $D_p(int(\Gamma),U)$ for all $T\in Ob_2(\Gamma)$. Let us define the function $u_{2,\Gamma}$ by the formula
$$u_{2,\Gamma}(T)=(u_{1,\Gamma}(ft(T)),u_{1,ft(T)}(T))$$
We can write this function as a composition of the bijection 
$$Ob_2(\Gamma)\sr \amalg_{\Gamma'\in Ob_1(\Gamma)}Ob_1(\Gamma')$$
that sends $T$ to $(ft(T),T)$ with the function
$$\amalg_{\Gamma'\in Ob_1(\Gamma)}Ob_1(\Gamma')\sr \amalg_{F\in Hom(int(\Gamma),U)}Hom((int(\Gamma);F),U)$$
that is the total function of the function $u_{1,\Gamma}$ and the family of functions $u_{1,\Gamma'}$ given for all $\Gamma'\in Ob_1(\Gamma)$. Since $u_{1,\Gamma}$ is a bijection and for each $\Gamma'$, $u_{1,\Gamma'}$ is a bijection, the total function is a bijection. 

Similarly, $(u_{1,\Gamma}(ft(\partial(o))),\wt{u}_{1,ft(\partial(o))}(o))$ is a well defined element of $D_p(int(\Gamma),\wt{U})$ since 
$$int(ft(\partial(o)))=(int(\Gamma);u_{1,\Gamma}(ft(\partial(o)))).$$ 

Define the function $\wt{u}_{2,\Gamma}$ be the formula
$$\wt{u}_{2,\Gamma}(o)=(u_{1,\Gamma}(ft(\partial(o))),\wt{u}_{1,ft(\partial(o))}(o))$$
We can write this function as the composition of the bijection
$$\wt{Ob}_2(\Gamma)\sr \amalg_{\Gamma'\in Ob_1(\Gamma)}\wt{Ob}_1(\Gamma')$$
that sends $o$ to $(ft(\partial(o)),o)$ with the function
$$\amalg_{\Gamma'\in Ob_1(\Gamma)}\wt{Ob}_1(\Gamma')\sr \amalg_{F\in Hom(int(\Gamma),U)}Hom((int(\Gamma);F),\wt{U})$$
that is the total function of the function $u_{1,\Gamma}$ and the family of functions $\wt{u}_{1,\Gamma'}$ given for all $\Gamma'\in Ob_1(\Gamma)$. Since $u_{1,\Gamma}$ is a bijection and for each $\Gamma'$, $\wt{u}_{1,\Gamma'}$ is a bijection, the total function is a bijection. 

The verification of the third and the fourth conditions of the problem are easy from the definition of $u_2$ and $\wt{u}_2$.
\end{construction}
\begin{remark}\rm
\llabel{2015.07.29.rem2}
The families of sets $D_p(X,V)$ together with the families of functions $D_p(f,V)$ and $D_p(X,g)$ define, as one can easily prove from definitions, a functor from ${\cal C}^{op}\times{\cal C}$ to $Sets$ or, if viewed as families $V\mapsto D_p(-,V)$, a functor
$$Yo_2:{\cal C}\sr PreShv({\cal C})$$
If $Yo_1=Yo$ is the Yoneda embedding then we can see $u_i$ for $i=1,2$ as isomorphisms
$$Ob_i\sr int_*(Yo_i(U))$$
and $\wt{u}_i$ as isomorphisms 
$$\wt{Ob}_i\sr int_*(Yo_i(\wt{U}))$$
These isomorphisms generalize easily to all $i>0$ if one defines, inductively,
$$Yo_{n+1}(V)(X)=\amalg_{F:X\sr U}Yo_n(V)((X;F))$$
Moreover, if we define $Hom_{n}(X,Y)$ as $Yo_n(Y)(X)$ then there are composition functions
$$Hom_n(X,Y)\times Hom_m(Y,Z)\sr Hom_{n+m}(X,Z)$$
that are likely to satisfy the unity and associativity axioms such that one obtains, from any universe category $({\cal C},p)$, a new category $({\cal C},p)_*$ with the same collection of objects and morphisms between two objects given by
$$Hom_{({\cal C},p)_*}(X,Y)=\amalg_{n\ge 1}Hom_m(X,Y)$$
In this paper we will not need $Yo_n$ for $n>2$ and we defer the study of this structure until the future papers.
\end{remark}

When $\cal C$ is a locally cartesian closed category (see appendix), the functors $D_p(-,V)$ become representable providing us with a way to describe operations such as $\Pi$ and $\lambda$ on $CC({\cal C},p)$ in terms of morphisms between objects in $\cal C$. 

For a morphism $p:\wt{U}\sr U$ in a locally cartesian closed category and an object $V$ of this category let 
$$I_p(V):=\uu{Hom}_U((\wt{U},p),(U\times V,pr_1))$$
and let 
$$prI_p(V)=p\triangle pr_1:I_p(V)\sr U$$
be the morphism that defines $I_p(V)$ as an object over $U$.

Note that $I_p$ depends on the choice of a locally cartesian closed structure on $\cal C$. On the other hand, the construction of the functors $D_p(X,V)$ requires a universe structure on $p$ but do not require a locally cartesian closed structure on $\cal C$. 

The computations below are required in order to establish the connections between the constructions that use the locally cartesian closed structure and the constructions that use universe structures. 

Let $p:\wt{U}\sr U$ be a universe and $V$ an object of $\cal C$. We assume that $\cal C$ is equipped with a locally cartesian closed structure. For $F:X\sr U$ there is a unique morphism
$$\iota_F:(X;F)\sr (X,f)\times_U(\wt{U},p)$$
such that $\iota_F\circ pr_1=p_{X,F}$ and $\iota_F\circ pr_2=Q(F)$ which is a particular case of the morphisms $\iota$, $\iota'$ of Lemma \ref{2015.04.16.l1}. 

The evaluation morphism in the case of $I_p(V)$ is of the form 
$$evI_p: (I_p(V),prI_p(V))\times_U(U\times V, pr_1)\sr U\times V$$
Define a morphism
$$st_p(V):(I_p(V);prI_p(V))\sr V$$
as the composition:
$$st_p(V):=\iota_{prI_p(V)}\circ evI_p(V)\circ pr_2$$
We will need to use some properties of these morphisms.
\begin{lemma}
\llabel{2015.04.14.l2a}
Let $f:V\sr V'$ be a morphism, then one has
$$Q(I_p(f),prI_p(V'))\circ st_p(V')=st_p(V)\circ f$$
\end{lemma}
\begin{proof}
Let $pr=prI_p(V)$, $pr'=prI_{p}(V')$, $\iota=\iota_{pr}$, $\iota'=\iota_{pr'}$, $ev=evI_p(V)$ and $ev'=evI_p(V')$. Then we have to verify that the outer square of the following diagram commutes:
$$
\begin{CD}
(I_p(V);pr) @>\iota>> (I_p(V),pr)\times_U(\wt{U},p) @>ev>> U\times V @>pr_2 >> V\\
@VQ(I_p(f),pr') VV @V I_p(f)\times Id_{\wt{U}} VV@V Id_U\times f VV @VV f V\\
(I_p(V');pr') @>\iota'>> (I_p(V'),pr')\times_U(\wt{U},p) @>ev'>> U\times V' @>pr_2 >> V'
\end{CD}
$$
The commutativity of the left square is a particular case of Lemma \ref{2015.04.16.l1}. The commutativity of the right square is an immediate corollary of the definition of $Id_U\times f$.  The commutativity of the middle square is a particular case of the axiom of locally cartesian closed structure that  says that morphisms $ev^{X}_{Y}$ are natural in $Y$. 
\end{proof}

\begin{problem}
\llabel{2015.03.29.prob1}
Let $({\cal C},p,pt)$ be a locally cartesian closed universe category. To construct, for all $X,V\in{\cal C}$, bijections
$$\eta_{X,V}:D_p(X,V)\sr Hom(X,I_p(V))$$
that are natural in $X$ and $V$, i.e., such that for any $d\in D_p(X,V)$ one has
\begin{enumerate}
\item for all $f:V\sr V'$ one has $\eta(d)\circ I_p(f)=\eta(D_p(X,f)(d))$,
\item for all $f:X'\sr X$ one has $f\circ \eta(d)=\eta(D_p(f,V)(d))$.
\end{enumerate}
\end{problem}
\begin{construction}\rm
\llabel{2015.03.29.constr1}
We will construct bijections 
$$\eta^{!}_{X,V}:Hom(X,I_p(V))\sr D_p(X,V)$$
such that for any $g:X\sr I_p(V)$ one has:
\begin{enumerate}
\item for all $f:V\sr V'$ one has $D_p(X,f)(\eta^!(g))=\eta^!(g\circ I_p(f))$,
\item for all $f:X'\sr X$ one has $D_p(f,V)(\eta^!(g))=\eta^!(f\circ g)$.
\end{enumerate}
and then define $\eta_{X,V}$ as the inverse to $\eta^!_{X,V}$. 

For $g:X\sr I_p(V)$ we set
$$\eta^{!}_{X,V}(g):=(g\circ prI_p(V), Q(g,prI_p(V))\circ st_p(V))$$
To see that this is a bijection observe first that it equals to the composition
$$Hom(X,I_p(V))\sr \amalg_{F:X\sr U}Hom_U((X,F),(I_p(V),prI_p(V)))\sr \amalg_{F:X\sr U}Hom((X;F),V)$$
where the first map is of the form $g\mapsto (g\circ prI_p(V),g)$ and the second is the sum over all $F:X\sr U$ of maps $g\mapsto Q(g,prI_p(V))\circ st_p(V)$. The first of these two maps is a bijection. It remains to show that the second one is a bijection for every $F$.

By definition of the $\uu{Hom}$ structure we know that for each $F$ the map
$$Hom_U((X,F),(I_{p}(V),prI_p(V)))\sr Hom_U(((X,F)\times_U(\wt{U},p),-),(U\times V,pr_1))$$
given by $g\mapsto (g\times Id_{\wt{U}})\circ evI_p(V)$ is a bijection. We also know that the map
$$Hom_U(((X,F)\times_U(\wt{U},p),F\dd p),(U\times V,pr_1))\sr Hom((X,F)\times_U(\wt{U},p), V)$$
is a bijection. Since $\iota_F$ is an isomorphism the composition with it is a bijection. Now we have two maps
$$Hom_U((X,F),(I_{p}(V),prI_p(V)))\sr Hom((X;F),V)$$
given by $g\mapsto \iota_F\circ (g\times Id_{\wt{U}})\circ evI_p(V)\circ p_V$ and $g\mapsto Q(g,prI_p(V))\circ st_p(V)$ of which the first one is the bijection. It remains to show that these maps are equal. For this it is sufficient to show that 
$$Q(g,prI_p(V))\circ \iota_{prI_p(V)}=\iota_F\circ(g\times Id_{\wt{U}})$$
which follows easily from computing compositions with the projections $pr_1$ to $I_p(V)$ and $pr_2$ to $\wt{U}$.

We now have to check the behavior of $\eta^!$ with respect to morphisms in $X$ and $V$.

Let $pr=prI_p(V)$ and $pr'=prI_p(V')$. For $f:V'\sr V$ and $f:X\sr I_p(V)$ we have
$$D_p(X,f)(\eta^!(g))=D_p(X,f)(g\circ pr, Q(g,pr)\circ st_p(V))=(g\circ pr, Q(g,pr)\circ st_p(V)\circ f)$$
and
$$\eta^!(g\circ I_p(f))=(g\circ I_p(f)\circ pr', Q(g\circ I_p(f),pr')\circ st_p(V'))$$
We have $pr=I_p(f)\circ pr'$ because $I_p(f)$ is a morphism over $U$. It remains to check that 
$$Q(g,pr)\circ st_p(V)\circ f=Q(g\circ I_p(f),pr')\circ st_p(V')$$
By \cite[Lemma 2.5]{Cfromauniverse} we have
$$Q(g\circ I_p(f),pr')=Q(g,pr)\circ Q(I_p(f),pr')$$
and the remaining equality
$$Q(g,pr)\circ st_p(V)\circ f=Q(g,pr)\circ Q(I_p(f),pr')\circ st_p(V')$$
follows from Lemma \ref{2015.04.14.l2a}.

Consider now $f:X'\sr X$. Then 
$$D_p(f,V)(\eta^!(g))=D_p(f,V)(g\circ pr, Q(g,pr)\circ st_p(V))=(f\circ g\circ pr,  Q(f, g\circ pr)\circ Q(g,pr)\circ st_p(V))$$
$$\eta^!(f\circ g)=(f\circ g\circ pr, Q(f\circ g, pr)\circ st_p(V))$$
and the required equality follows from \cite[Lemma 2.5]{Cfromauniverse}.  
\end{construction}
\begin{problem}
\llabel{2015.03.17.prob3}
For a locally cartesian closed closed $\cal C$ and a universe $p:\wt{U}\sr U$ in $\cal C$ to construct for any $\Gamma\in Ob(CC({\cal C},p))$ bijections
$$\mu_{2,\Gamma}: Ob_2(\Gamma)\sr Hom_{\cal C}(int(\Gamma),I_p(U))$$
and
$$\wt{\mu}_{2,\Gamma}: \wt{Ob}_2(\Gamma)\sr Hom_{\cal C}(int(\Gamma),I_p(\wt{U}))$$
that are natural in $\Gamma$ and such that with respect to these bijections $\partial$ corresponds to composition with $I_p(p)$.
\end{problem}
\begin{construction}
\llabel{2015.03.17.constr2}\rm
Compose bijections $u_2$ and $\wt{u}_2$ with the bijection $\eta$ of Construction \ref{2015.03.29.constr1} in the case $V=U$ and $V=\wt{U}$ respectively. 
\end{construction}
\begin{remark}
\rm\llabel{2015.03.29.rem2}
The previous constructions related to $Ob_2$ and $\wt{Ob}_2$ can be easily generalized to $Ob_n$ and $\wt{Ob}_n$  for all $n>0$. For example there are natural bijections
$$\mu_{n+1}:Ob_{n+1}(\Gamma)\sr Hom(int(\Gamma), I_p^n(U))$$
$$\wt{\mu}_{n+1}:\wt{Ob}_{n+1}(\Gamma)\sr Hom(int(\Gamma), I_p^n(\wt{U}))$$
where $I_p^n$ is the n-th iteration of the functor $I_p$ and $\mu_1=u_1$ and $\wt{\mu}_1=\wt{u}_1$. More generally, the functors $Yo_n(V)$ of Remark \ref{2015.07.29.rem2} in the case of a locally cartesian closed universe category $({\cal C},p)$ are representable by objects $I_p^n(V)$. 
\end{remark}

\subsection{$(\Pi,\lambda)$-structures on the C-systems $CC({\cal C},p)$}
We will show now how to construct $(\Pi,\lambda)$-structures on C-systems of the form $CC({\cal C},p)$ for locally cartesian closed (pre-)categories\footnote{For the discussion of the difference between a category and a pre-category see the introduction to \cite{Csubsystems} and \cite{RezkCompletion}.} $\cal C$.

\begin{definition}
\llabel{2015.03.29.def1}
Let $\cal C$ be a locally cartesian closed category, $pt$ be a final object in $\cal C$ and $p:\wt{U}\sr U$ a universe. A $\Pi$-structure on $p$ is a pair of morphisms 
$$\wt{P}:I_p(\wt{U}) \sr \wt{U}$$
$$P:I_p(U) \sr U$$
such that the square
\begin{eq}
\llabel{2009.prod.square}
\begin{CD}
I_p(\wt{U}) @>\wt{P}>> \wt{U}\\
@VV I_p(p) V @VV p V\\
I_p(U) @>P>> U
\end{CD}
\end{eq}
is a pull-back square.
\end{definition}
\begin{problem}
\llabel{2015.03.17.prob0}
Let $\cal C$ be a locally cartesian closed category, $pt$ be a final object in $\cal C$ and $p:\wt{U}\sr U$ a universe. Let $(\wt{P},P)$ be a $\Pi$-structure on $p$. To construct a $(\Pi,\lambda)$-structure on $CC({\cal C},p)$.
\end{problem}
\begin{construction}
\llabel{2015.03.17.constr3}\rm
Let $\Gamma\in Ob(CC({\cal C},p))$. For $T\in Ob_2(\Gamma)$ set
$$\Pi_{\Gamma}(T)=u_1^{-1}(u(T)\circ P)$$
and for $s\in \wt{Ob}_2(\Gamma)$ set
$$\lambda_{\Gamma}(s)=\wt{u}_1^{-1}(\wt{u}_2(s)\circ \wt{P})$$
These gives us maps
$$\Pi_{\Gamma}:Ob_2(\Gamma)\sr Ob_1(\Gamma)$$
$$\lambda_{\Gamma}:\wt{Ob}_2(\Gamma)\sr \wt{Ob}_1(\Gamma)$$
The naturality of $u$ and $\wt{u}_2$ relative to morphisms $f:\Gamma'\sr \Gamma$ implies that these maps are natural with respect to such morphisms i.e. the squares (\ref{2015.03.09.sq1}) and (\ref{2015.03.09.sq2}) of Definition \ref{2015.03.09.def1} commute. One also verifies easily that $\partial(\lambda_{\Gamma}(s))=\Pi_{\Gamma}(\partial(s))$.

To verify that this pre-$(\Pi,\lambda)$-structure satisfies the Definition \ref{2015.03.09.def2} of $(\Pi,\lambda)$-structure one verifies that the bijections $\wt{u}_2$, $u_2$, $\wt{u}_1$ and $u_1$ define an isomorphism from the square (\ref{2015.03.09.eq1}) to the square obtained from (\ref{2009.prod.square}) by taking Hom-sets $Hom(int(\Gamma),-)$. Since the later square is pull-back and a square isomorphic to a pull-back square is a pull-back square the square (\ref{2015.03.09.eq1}) is a pull-back square and $(\Pi,\lambda)$ is a $(\Pi,\lambda)$-structure. 
\end{construction}

\comment{
\begin{definition}
\llabel{2009.10.27.def1}
Let $\cal C$ be an lcc category and let $p_i:\wt{U}_i\sr U_i$, $i=1,2,3$ be three morphisms in $\cal C$. A $\Pi$-structure on $(p_1,p_2,p_3)$ is a Cartesian square of the form
\begin{eq}
\llabel{Pisq1}
\begin{CD}
\uu{Hom}_{U_1}(\wt{U}_1,U_1\times \wt{U_2}) @>\wt{P}>> \wt{U}_3\\
@Vp_2'VV @VVp_3V\\
\uu{Hom}_{U_1}(\wt{U}_1,U_1\times U_2) @>P>> U_3
\end{CD}
\end{eq}
such that $p_2'$ is the natural morphism defined by $p_2$. A $\Pi$-structure on $p:\wt{U}\sr U$ is a $\Pi$-structure on $(p,p,p)$.
\end{definition}
Let $\cal C$ be as above, $p:\wt{U}\sr U$ and let $(\wt{P},P)$ be a $\Pi$-structure on $(p,p,p)$. Let us construct a $(\Pi,\lambda)$-structure on the C-system $CC=CC({\cal C},p)$. 
}

\subsection{More on universe category functors I}

Let $({\cal C},p,pt)$ and $({\cal C},p',pt')$ be two universe (pre-)categories. Recall from \cite{Cfromauniverse} that a functor of universe categories from $({\cal C},p,pt)$ to  $({\cal C},p',pt')$  is a triple ${\bf\Phi}=(\Phi,\phi,\wt{\phi})$ where $\Phi$ is a functor ${\cal C}\sr {\cal C}'$ and $\phi:\Phi(U)\sr U'$, $\wt{\phi}:\Phi(\wt{U})\sr \wt{U}'$ are two morphisms such that $\Phi$ takes the final object to a final object, pull-back squares based on $p$ to pull-back squares and such that the square 
\begin{eq}
\llabel{2015.03.21.sq1}
\begin{CD}
\Phi(\wt{U}) @>\wt{\phi}>> \wt{U}'\\
@V\Phi(p)VV @VVp' V\\
\Phi(U) @>\phi>> U'
\end{CD}
\end{eq}
is a pull-back square.
 
For $X,V$ in $\cal C$ we have the functoriality map
$$\Phi:Hom(X,V)\sr Hom(\Phi(X),\Phi(V))$$
\begin{problem}
\llabel{2015.04.12.prob1}
For a universe category functor ${\bf\Phi}=(\Phi,\phi,\wt{\phi})$, to define, for all $X, V\in {\cal C}$, functions
$${\bf \Phi}^2:D_p(X,V)\sr D_{p'}(\Phi(X),\Phi(V))$$
\end{problem}
\begin{construction}\rm
\llabel{2015.04.12.constr1}
Let $(F_1:X\sr U, F_2:(X;F_1)\sr V)$ be an element in $D_p(X,V)$. 
Consider $(\Phi(X);\Phi(F_1)\circ \phi)$. Since the square (\ref{2015.03.21.sq1}) is a pull-back square there is a unique morphism $q$ such that $q\circ \wt{\phi}=Q(\Phi(F_1)\circ \phi)$ and $q\circ \Phi(p)=p_{\Phi(X),\Phi(F_1)\circ \phi}\circ \Phi(F_1)$ and then the left hand side square in the diagram
$$
\begin{CD}
(\Phi(X);\Phi(F_1)\circ \phi) @>q>> \Phi(\wt{U}) @>\wt{\phi}>> \wt{U}'\\
@VV p_{\Phi(X),\Phi(F_1)\circ \phi} V @V\Phi(p) VV @VV p' V\\
\Phi(X) @>\Phi(F_1)>> \Phi(U) @>\phi>> U'
\end{CD}
$$
is a pull-back square. Together with the fact that $\Phi$ takes pull-back squares based on $p$ to pull-back squares we obtain a unique morphism, which is an isomorphism,
$$\iota:(\Phi(X);\Phi(F_1)\circ \phi)\sr \Phi(X;F_1)$$
such that 
\begin{eq}\llabel{2015.04.08.eq1}
\iota\circ \Phi(p_{X,F_1})=p_{\Phi(X),\Phi(F_1)\circ \phi}
\end{eq}
\begin{eq}\llabel{2015.04.08.eq2}
\iota\circ \Phi(Q(F_1))\circ\wt{\phi}=Q(\Phi(F_1)\circ\phi)
\end{eq}
and we define:
$${\bf\Phi}^2(F_1,F_2):=(\Phi(F_1)\circ \phi, \iota\circ \Phi(F_2))$$
\end{construction}
We will need the following properties of the maps below.
\begin{lemma}
\llabel{2015.03.23.l1}
Let $\Phi$ be as above, $f:X'\sr X$ be a morphism and $V$ be an object of $\cal C$. Then the square 
$$
\begin{CD}
D_p(X,V) @>D_p(f,V)>> D_p(X',V)\\
@V{\bf \Phi}^2VV @V{\bf \Phi}^2 VV\\
D_{p'}(\Phi(X),\Phi(V)) @>D_{p'}(\Phi(f),\Phi(V))>> D_{p'}(\Phi(X'),\Phi(V))
\end{CD}
$$
commutes.
\end{lemma}
\begin{proof}
We have to show that for any $d\in D_p(X,V)$ one has 
$$D_{p'}(\Phi(f),\Phi(V))({\bf\Phi}^2(d))={\bf\Phi}^2(D_p(f,V)(d))$$
Let $d=(F_1,F_2)$. Then 
$$D_{p'}(\Phi(f),\Phi(V))({\bf\Phi}^2(d))=D_{p'}(\Phi(f),\Phi(V))(\Phi(F_1)\circ\phi,\iota\circ\Phi(F_2))=$$
$$(\Phi(f)\circ \Phi(F_1)\circ \phi, q'\circ \iota\circ \Phi(F_2))$$
and 
$${\bf\Phi}^2(D_p(f,V)(F_1,F_2))={\bf\Phi}^2(f\circ F_1,q\circ F_2)=$$
$$(\Phi(f\circ F_1)\circ\phi,\iota'\circ\Phi(q\circ F_2))$$
where
$$\iota:(\Phi(X);\Phi(F_1)\circ\phi)\sr \Phi(X;F_1)\spc\spc \iota':(\Phi(X');\Phi(f\circ F_1)\circ\phi)\sr \Phi(X';f\circ F_1)$$
$$q:(X';f\circ F_1)\sr (X;F_1)\spc\spc q':(\Phi(X');\Phi(f)\circ \Phi(F_1)\circ\phi)\sr (\Phi(X);\Phi(F_1)\circ\phi)$$
are the morphisms defined in Construction \ref{2015.04.12.constr1}. We have
$$\Phi(f)\circ \Phi(F_1)\circ \phi=\Phi(f\circ F_1)\circ\phi$$
and it remains to check that
$$q'\circ \iota\circ \Phi(F_2)=\iota'\circ\Phi(q\circ F_2)$$
or that $q'\circ \iota = \iota'\circ \Phi(q)$. The codomain of both morphisms is $\Phi(X;F_1)$ that by our assumption on $\Phi$ is a pull-back of $p'$ and $\Phi(F_1)\circ\phi$. Therefore it is sufficient to verify that the compositions of these two morphisms with the projections to $\wt{U}'$ and $\Phi(X)$ coincide.

This is done by a direct computation from definitions. 
\end{proof}
\begin{lemma}
\llabel{2015.04.10.l3}
Let ${\bf \Phi}$ be as above, $X$ an object of $\cal C$ and $f:V\sr V'$ a morphism. Then the square
$$
\begin{CD}
D_p(X,V) @>D_p(X,f)>> D_p(X,V')\\
@V{\bf\Phi}^2 VV @VV{\bf \Phi}^2V\\
D_{p'}(\Phi(X),\Phi(V)) @>D_p(\Phi(X),\Phi(f))>> D_{p'}(\Phi(X),\Phi(V'))
\end{CD}
$$
commutes.
\end{lemma}
\begin{proof}
Let $d=(F_1,F_2)\in D_p(X,V)$. We have to show that
$${\bf\Phi}^2(D_p(X,f)(F_1,F_2))=D_p(\Phi(X),\Phi(f))({\bf \Phi}^2(F_1,F_2))$$
We have:
$${\bf \Phi}^2(D_p(X,f)(F_1,F_2))={\bf \Phi}^2((F_1,F_2\circ f))=(\Phi(F_1)\circ \phi, \iota\circ \Phi(F_2\circ f))=$$$$(\Phi(F_1)\circ \phi, \iota\circ \Phi(F_2)\circ \Phi(f))=D_p(\Phi(X),\Phi(f))({\bf \Phi}^2(F_1,F_2))$$
\end{proof}
Note that in the problem below no assumption is made about the compatibility of $\Phi$ with the locally cartesian closed structures on $\cal C$ and ${\cal C}'$. 
\begin{problem}
\llabel{2015.03.21.prob1}
Assume that $\cal C$ and ${\cal C}'$ are locally cartesian closed universe categories. For ${\bf\Phi}$ as above and $V\in{\cal C}$ to construct  a morphism
$$\chi_{\bf\Phi}(V):\Phi(I_p(V))\sr I_{p'}(\Phi(V))$$
\end{problem}
\begin{construction}\rm
\llabel{2015.03.21.constr1}
Let 
$$\eta:D_p(X,V)\sr Hom(X,I_p(V))$$
$$\eta':D_{p'}(X',V')\sr Hom(X',I_{p'}(V'))$$
be bijections from Construction \ref{2015.03.29.constr1}. We define:
$$\chi_{\bf\Phi}(V):=\eta'({\bf\Phi}^2(\eta^{!}(Id_{I_p(V)})))$$
for $X=I_p(V)$ and $X'=\Phi(I_p(V))$.
\end{construction}
Let us show that $\chi_{\bf\Phi}$ are natural in $V$.
\begin{lemma}
\llabel{2015.04.10.l4}
For ${\bf\Phi}$ as above let $f:V_1\sr V_2$ be a morphism. Then the square
$$
\begin{CD}
\Phi(I_p(V_1)) @>\chi(V_1)>> I_{p'}(\Phi(V_1))\\
@V\Phi(I_p(f)) VV @VV I_{p'}(\Phi(f)) V\\
\Phi(I_p(V_2)) @>\chi(V_2)>> I_{p'}(\Phi(V_2))
\end{CD}
$$
commutes.
\end{lemma}
\begin{proof}
We have:
$$\chi(V_1)\circ I_{p'}(\Phi(V_1))=\eta'({\bf \Phi}^2(\eta^{!}(Id_{X_1})))\circ I_{p'}(\Phi(f))=\eta'(D_p(X_1,\Phi(f))({\bf\Phi}^2(\eta^{!}(Id_{X_1}))))$$
where $X=I_p(V_1)$, by naturality of $\eta'$. Then
$$\eta'(D_p(X_1,\Phi(f))({\bf\Phi}^2(\eta^{!}(Id_{X_1}))))=\eta'({\bf \Phi}^2(D_p(X_1,f)(\eta^{!}(Id_{X_1}))))=$$
$$\eta'({\bf \Phi}^2(\eta^{!}(Id_{X_1}\circ I_p(f)))=\eta'({\bf \Phi}^2(\eta^{!}(I_p(f))))$$
where the first equality holds by Lemma \ref{2015.04.10.l3} and the second by Problem \ref{2015.03.29.prob1}(1).

On the other hand:
$$\Phi(I_p(f))\circ \chi(V_2)=\Phi(I_p(f))\circ \eta'({\bf\Phi}^2(\eta^{!}(Id_{X_2})))=$$$$\eta'(D_{p'}(\Phi(I_p(f)),\Phi(X_2))({\bf\Phi}^2(\eta^{!}(Id_{X_2}))))$$
by naturality of $\eta'$. Then
$$\eta'(D_{p'}(\Phi(I_p(f)),\Phi(X_2))({\bf\Phi}^2(\eta^{!}(Id_{X_2}))))=\eta'({\bf\Phi}^2(D_p(I_p(f),X_2)(\eta^{!}(Id_{X_2}))))=$$
$$\eta'({\bf\Phi}^2(\eta^{!}(I_p(f)\circ Id_{X_2})))=\eta'({\bf\Phi}^2(\eta^{!}(I_p(f))))$$
where the first equality holds by Lemma \ref{2015.04.10.l3} and the second by Problem \ref{2015.03.29.prob1}(2). This finishes the proof of Lemma \ref{2015.04.10.l4}.
\end{proof}
\begin{lemma}
\llabel{2015.05.06.l1}
For all $X,V\in{\cal C}$ and $a\in D_p(X,V)$ one has
$$\Phi(\eta(a))\circ \chi_{\bf\Phi}(V)=\eta'({\bf\Phi}^2(a))$$
\end{lemma}
\begin{proof}
By definition of $\chi_{\bf\Phi}$ and contravariant functoriality of $\eta'$ we have
$$\Phi(\eta(a))\circ \chi_{\bf\Phi}(V)=\Phi(\eta(a))\circ \eta'({\bf\Phi}^2(\eta^{!}(Id)))=\eta'(D_{p'}(\Phi(\eta(a)),\Phi(V))({\bf\Phi}^2(\eta^{!}(Id_{I_p(V)}))))$$
By Lemma \ref{2015.03.23.l1} we further have:
$$\eta'(D_{p'}(\Phi(\eta(a)),\Phi(V))({\bf\Phi}^2(\eta^{!}(Id))))=\eta'({\bf\Phi}^2(D_p(\eta(a),V)(\eta^{!}(Id))))$$
It remains to show that $D_p(\eta(a),V)(\eta^{!}(Id))=f$. Since $\eta$ is a bijection we may apply it on both sides and by functoriality of $\eta$ we get
$$\eta(D_p(\eta(a),V)(\eta^{!}(Id)))=\eta(f)\circ \eta(\eta^{!}(Id))=\eta(f)\circ Id=\eta(f).$$
\end{proof}

\subsection{More on universe category functors II}

By \cite[Construction 4.7]{Cfromauniverse} any universe category functor ${\bf \Phi}=(\Phi,\phi,\wt{\phi})$ defines a homomorphism of C-systems
$$H:CC({\cal C},p)\sr CC({\cal C}',p')$$
Let $\psi:pt'\sr \Phi(pt)$ be the unique morphism. To define $H$ on objects, one uses the fact that
$$Ob(CC({\cal C},p))=\amalg_{n\ge 0} Ob_n({\cal C},p)$$
and defines $H(n,A)$ as $(n,H_n(A))$ where
$$H_n:Ob_n({\cal C},p)\sr Ob_n({\cal C}',p')$$
To obtain $H_n$ one defines by induction on $n$, pairs $(H_n,\psi_n)$ where $H_n$ is as above and $\psi_n$ is a family of isomorphisms 
$$\psi_{n}(A):int'(H_n(A))\sr \Phi(int(A))$$
as follows:
\begin{enumerate}
\item for $n=0$, $H_0$ is the unique map from one point set to one point set and $\psi_{0}(A)=\psi$,
\item for the successor of $n$ one has 
$$H_{n+1}(A,F)=(H_n(A),\psi_n(A)\circ\Phi(F)\circ \phi)$$
and $\psi_{n+1}{A,F}$ is the unique morphism $int'(H(A,F))\sr \Phi(int(A,F))$ such that
$$\psi(A,F)\circ \Phi(Q(F))\circ\wt{\phi}=Q'(\psi(A)\circ\Phi(F)\circ\phi)$$
and
$$\psi(A,F)\circ \Phi(p_(A,F))=p_{H(A,F)}\circ \psi(A)$$
\end{enumerate}
The action of $H$ on morphisms is given, for $f:(m,A)\sr(n,B)$, by
$$H(f)=\psi(A)\circ\Phi(int(f))\circ\psi(B)^{-1}$$
We will often write $H$ also for the functions $H_n$ and $\psi$ for the functions $\psi_n$.

Let $\Gamma\in Ob(CC({\cal C},p))$ and consider the bijections of Constructions \ref{2015.05.02.constr1} and \ref{2015.05.02.constr2}. 

In order to prove our main functoriality Theorem \ref{2015.03.21.th1} we need describe in more detail the maps 
$$Ob_1(\Gamma)\sr Ob_1(H(\Gamma))$$
$$Ob_2(\Gamma)\sr Ob_2(H(\Gamma))$$
and the similar maps on $\wt{Ob}_1$ and $\wt{Ob}_2$ that are defined by $H$.

\begin{lemma}
\llabel{2015.03.21.l4}
Let $(\Phi,\phi,\wt{\phi})$ be universe category functor. Then:
\begin{enumerate}
\item for $T\in Ob_1(\Gamma)$ one has 
$$u_{1,H(\Gamma)}(H(T))=\psi(\Gamma)\circ\Phi(u_{1,\Gamma}(T))\circ\phi$$
\item for $o\in \wt{Ob}_1(\Gamma)$ one has
$$\wt{u}_{1,H(\Gamma)}(H(o))=\psi(\Gamma)\circ\Phi(\wt{u}_{1,\Gamma}(o))\circ\wt{\phi}$$ 
\item for $T\in Ob_2(\Gamma)$ one has
$$u_{2,H(\Gamma)}(H(T))=D_{p'}(\psi(\Gamma),U')(D_{p'}(int'(H(\Gamma)),\phi)({\bf\Phi}^2(u_{2,\Gamma}(T))))$$ 
\item for $o\in\wt{Ob}_2(\Gamma)$ one has
$$\wt{u}_{2,H(\Gamma)}(H(o))=D_{p'}(\psi(\Gamma),\wt{U}')(D_{p'}(int'(H(\Gamma)),\wt{\phi})({\bf\Phi}^2(\wt{u}_{2,\Gamma}(o))))$$
\end{enumerate}
\end{lemma}
\begin{proof}
Let $\Gamma=(n,A)$.

In the case of $T\in Ob_1(\Gamma)$, if $T=(n+1,(A,F)$ then
$$u_1(H(T))=u_1(n+1,H(A,F))=u_1(n+1,(H(A), \psi(\Gamma)\circ\Phi(F)\circ\phi))=\psi(\Gamma)\circ\Phi(F)\circ\phi$$
In the case of $s\in \wt{Ob}_1(\Gamma)$, if $F=u_1(\partial(s))$ then
$$\wt{u}_1(H(s))=H(s)\circ Q'(u_1(n+1,H(A,F)))=\psi(A)\circ \Phi(s)\circ \psi(A,F)^{-1}\circ Q'(\psi(A)\circ\Phi(F)\circ\phi)=$$
$$\psi(A)\circ \Phi(s)\circ \Phi(Q(F))\circ\wt{\phi}=\psi(A)\circ \Phi(s\circ Q(F))\circ \wt{\phi}=\psi(A)\circ\Phi(\wt{u}_1(s))\circ\wt{\phi}$$
In the case $T\in Ob_2(\Gamma)$, if $T=(n+2,((A,F_1),F_2))$ then
$$u_2(H(T))=u_2(n+2,H(((A,F_1),F_2)))=u_2(n+2,(H(A,F_1),\psi(A,F_1)\circ \Phi(F_2)\circ \phi))=$$
$$u_2(n+2,(H(A),\psi(A)\circ\Phi(F_1)\circ\phi,\psi(A,F_1)\circ\Phi(F_2)\circ\phi))=$$
$$(\psi(A)\circ\Phi(F_1)\circ\phi,\psi(A,F_1)\circ\Phi(F_2)\circ\phi)$$
On the other hand 
$$D_{p'}(\psi(A),-)D_{p'}(-,\phi)({\bf\Phi}^2(u_2(T)))=D_{p'}(\psi(A),-)D_{p'}(-,\phi)({\bf\Phi}^2(u_2(n+2,((A,F_1),F_2))))=$$
$$D_{p'}(\psi(A),-)D_{p'}(-,\phi)({\bf\Phi}^2(F_1,F_2))=D_{p'}(\psi(A),-)D_{p'}(-,\phi)(\Phi(F_1)\circ\phi,\iota\circ\Phi(F_2))=$$
$$D_{p'}(\psi(A),-)(\Phi(F_1)\circ\phi,\iota\circ\Phi(F_2)\circ\phi)=(\psi(A)\circ \Phi(F_1)\circ\phi, Q'(\psi(A),\Phi(F_1)\circ\phi)\circ\iota\circ\Phi(F_2)\circ\phi)$$
therefore we need to show that
\begin{eq}
\llabel{2015.04.12.eq1}
\psi(A,F_1)\circ\Phi(F_2)\circ\phi=Q'(\psi(A),\Phi(F_1)\circ\phi)\circ\iota\circ\Phi(F_2)\circ\phi
\end{eq}
Using the fact that the external square of the diagram
$$
\begin{CD}
\Phi(int(A,F_1)) @>\Phi(Q(F_1))>> \Phi(\wt{U}) @>\wt{\phi}>> \wt{U}'\\
@V\Phi(p_{(A,F_1)})VV @VV\Phi(p)V @VVp'V\\
\Phi(int(A)) @>\Phi(F_1)>> \Phi(U) @>\phi>> U'
\end{CD}
$$
is a pull-back square we see that equality (\ref{2015.04.12.eq1}) would follow from the following two equalities: 
$$\psi(A,F_1)\circ \Phi(Q(F_1))\circ \wt{\phi}=Q'(\psi(A),\Phi(F_1)\circ\phi)\circ\iota\circ \Phi(Q(F_1))\circ \wt{\phi}$$
and
$$\psi(A,F_1)\circ \Phi(p_{(A,F_1)})=Q'(\psi(A),\Phi(F_1)\circ\phi)\circ\iota\circ \Phi(p_{(A,F_1)})$$
For the first equality we have
$$\psi(A,F_1)\circ \Phi(Q(F_1))\circ \wt{\phi}=Q'(\psi(A)\circ\Phi(F_1)\circ \phi)$$
by definition of $\psi(\Gamma,F_1)$ and
$$Q'(\psi(A),\Phi(F_1)\circ\phi)\circ\iota\circ \Phi(Q(F_1))\circ \wt{\phi}=Q'(\psi(A),\Phi(F_1)\circ\phi)\circ Q'(\Phi(F_1)\circ \phi)=Q'(\psi(A)\circ\Phi(F_1)\circ \phi)$$
where the first equality holds by definition of $\iota$ and second by the definition of $Q(-,-)$. 

For the second equality we have 
$$\psi(A,F_1)\circ \Phi(p_{(A,F_1)})=p_{H(A,F_1)}\circ \psi(A)$$
by definition of $\psi(A,F_1)$ and
$$Q'(\psi(A),\Phi(F_1)\circ\phi)\circ\iota\circ \Phi(p_{(A,F_1)})=Q'(\psi(A),\Phi(F_1)\circ\phi)\circ p_{\Phi(int(A)),\Phi(F_1)\circ\phi}=p_{H(A,F_1)}\circ\psi_{\Gamma}$$
by definitions of $Q'$ and $\iota$.

The case of $s\in \wt{Ob}_2(\Gamma)$ is strictly parallel to the case of $T\in Ob_2(\Gamma)$ with $\Phi(F_2)\circ\phi$ at the end of the formulas replaced by $\Phi(F_2')\circ \wt{\phi}$ where instead of $F_2:int(A,F_1)\sr U$ one has $F_2':int(A,F_1)\sr\wt{U}$.  
\end{proof}
For $(\Phi,\phi,\wt{\phi})$ as above let us denote by 
$$\xi_{\bf\Phi}:\Phi(I_p(U))\sr I_{p'}(U')$$
the composition $\chi_{\bf\Phi}(U)\circ I_{p'}(\phi)$ and by
$$\wt{\xi}_{\bf\Phi}:\Phi(I_p(\wt{U}))\sr I_{p'}(\wt{U}')$$
the composition $\chi_{\bf\Phi}(\wt{U})\circ I_p(\wt{\phi})$. 

\begin{lemma}
\llabel{2015.05.06.l2}
Let $(\Phi,\phi,\wt{\phi})$ be a universe category functor and $\Gamma\in Ob(CC({\cal C},p))$. Then one has:
\begin{enumerate}
\item for $T\in Ob_2(\Gamma)$
$$\eta_{p'}(u_2'(H(T)))=\psi(\Gamma)\circ \Phi(\eta_p(u_2(T)))\circ \xi_{\bf\Phi}$$
\item for $s\in \wt{Ob}_2(\Gamma)$
$$\eta_{p'}(\wt{u}_2'(H(s)))=\psi(\Gamma)\circ \Phi(\eta_p(\wt{u}_2(s)))\circ \wt{\xi}_{\bf\Phi}$$
\end{enumerate}
\end{lemma}
\begin{proof}
We have
$$\eta_{p'}(u_2'(H(T)))=\eta_{p'}(D_{p'}(\psi(\Gamma),\_)(D_{p'}(\_,\phi)({\bf\Phi}^2(u_2(T)))))=\psi(\Gamma)\circ\eta_{p'}({\bf \Phi}^2(u_2(T)))\circ I_{p'}(\phi)$$
where the first equality holds by Lemma \ref{2015.03.21.l4}(3) and the second by the naturality of $\eta_{p'}$. Next
$$\eta_{p'}({\bf \Phi}^2(u_2(T)))\circ I_{p'}(\phi)=\Phi(\eta(u_2(T)))\circ\chi_{\Phi}(U)\circ I_{p'}(\phi)=\Phi(\eta(u_2(T)))\circ \xi_{\bf\Phi}$$
where the first equality holds by Lemma \ref{2015.05.06.l1} and the second one by the definition of $\xi_{\bf\Phi}$. 

The proof of the second part of the lemma is strictly parallel to the proof of the first part. 
\end{proof}

\subsection{Functoriality properties of the $(\Pi,\lambda)$-structures arising from universes}

Let us prove the functoriality properties of the $(\Pi,\lambda)$ structures of Construction \ref{2015.03.17.constr3}.

The notion of a homomorphism of C-systems with $(\Pi,\lambda)$-structures used in the theorem below is defined in the obvious way. 
\begin{theorem}
\llabel{2015.03.21.th1}
Let $(\Phi,\phi,\wt{\phi})$ be as above and let $(P,\wt{P})$, $(P',\wt{P}')$ be as in Problem \ref{2015.03.17.prob0} for $\cal C$ and $\cal C'$ respectively.

Assume that the squares
\begin{eq}\llabel{2015.03.23.sq1a}
\begin{CD}
\Phi(I_p(U)) @>\xi_{\bf\Phi}>>  I_{p'}(U')\\
@V\Phi(P) VV @VV P' V\\
\Phi(U) @>\phi>> U
\end{CD}
\end{eq}
and
\begin{eq}\llabel{2015.03.23.sq1b}
\begin{CD}
\Phi(I_p(\wt{U})) @>\wt{\xi}_{\bf\Phi}>>  I_{p'}(\wt{U}')\\
@V\Phi(\wt{P}) VV @VV \wt{P}' V\\
\Phi(\wt{U}) @>\wt{\phi}>> \wt{U}
\end{CD}
\end{eq}
commute. Then the homomorphism 
$$H(\Phi,\phi,\wt{\phi}):CC({\cal C},p)\sr CC ({\cal C}',p')$$
is a homomorphism of C-systems with $(\Pi,\lambda)$-structures.
\end{theorem}
\begin{proof}
We have to show that for all $\Gamma\in Ob(CC({\cal C},p))$ and $T\in Ob_2(\Gamma)$ we have
$$\Pi'(H(T))=H(\Pi(T))$$
and for all $\Gamma\in Ob(CC({\cal C},p))$ and $s\in \wt{Ob}_2(\Gamma)$ we have
$$\lambda'(H(s))=H(\lambda(s))$$
We will prove the first equality. The proof of the second is strictly parallel to the proof of the first.

By definition we have:
$$\Pi'(H(T))=(u_1')^{-1}(u_2'(H(T))\circ P')=(u_1')^{-1}(\eta'(u_2'(H(T)))\circ P')$$
and
$$H(\Pi(T))=H(u_1^{-1}(\eta(u_2(T))\circ P))=(u_1')^{-1}(\psi(\Gamma)\circ \Phi(\eta(u_2(T))\circ P)\circ \phi)=$$$$(u_1')^{-1}(\psi(\Gamma)\circ \Phi(\eta(u_2(T)))\circ \Phi(P)\circ \phi)$$
where the second equality holds by Lemma \ref{2015.03.21.l4}(1). Let us show that 
$$\eta'(u_2'(H(T)))\circ P'=\psi(\Gamma)\circ \Phi(\eta(u_2(T)))\circ \Phi(P)\circ \phi$$
By Lemma \ref{2015.05.06.l2}(1) we have
$$\eta'(u_2'(H(T)))\circ P'=\psi(\Gamma)\circ\Phi(\eta(u_2(T)))\circ \xi_{\Phi}\circ P'$$
It remains to show that 
$$\xi_{\Phi}\circ P'=\Phi(P)\circ \phi$$
which is our assumption about the commutativity of the square (\ref{2015.03.23.sq1a}). 
\end{proof}
%


\subsection{Appendix: some constructions and theorems about categories}
\begin{lemma}
\llabel{2015.04.16.l1}
Let $\cal C$ be a category. Consider four fiber squares
$$
\begin{CD}
pb_i @>pr_{Y,i}>> Y\\
@Vpr_{X,i}VV @VVgV\\
X @>f>> Z
\end{CD}
\spc\spc
\begin{CD}
pb'_i @>pr_{Y',i}>> Y'\\
@Vpr_{X,i}VV @VVg'V\\
X' @>f'>> Z
\end{CD}
$$
where $i=1,2$. Let $a:X'\sr X$ and $b:Y'\sr Y$ be such that $a\circ f=f'$ and $b\circ g=g'$. Let $\iota:pb_1\sr pb_2$ be the unique morphism such that $\iota\circ pr_{X_2}=pr_{X,1}$ and $\iota\circ pr_{Y,1}=pr_{Y,2}$ and similarly for $\iota':pb_1'\sr pb_2'$. Let $pb_i(a,b):pb_i'\sr pb_i$ be the unique morphisms such that $pb_i(a,b)\circ pr_{X,i}=pr_{X',i}\circ a$ and $pb_i(a,b)\circ pr_{Y,i}=b\circ pr_{Y',i}$. Then the square
$$
\begin{CD}
pb_1' @>pb_1(a,b)>> pb_1\\
@V\iota' VV @VV\iota V\\
pb_2' @>pb_2(a,b)>> pb_2
\end{CD}
$$
commutes, i.e., $pb_1(a,b)\circ \iota=\iota'\circ pb_2(a,b)$.
\end{lemma}
\begin{proof}
Since $pb_2$ is a fiber product it is sufficient to prove that
$$pb_1(a,b)\circ \iota\circ pr_{X,2}=\iota'\circ pb_2(a,b)\circ pr_{X,2}$$
and
$$pb_1(a,b)\circ \iota\circ pr_{Y,2}=\iota'\circ pb_2(a,b)\circ pr_{Y,2}$$
For the first one we have:
$$pb_1(a,b)\circ \iota\circ pr_{X,2}=pb_1(a,b)\circ pr_{X,1}=pr_{X',1}\circ a$$
and
$$\iota'\circ pb_2(a,b)\circ pr_{X,2}=\iota'\circ pr_{X',2}\circ a=pr_{X',1}\circ a$$
The verification of the second equality is similar.
\end{proof}
\begin{definition}
\llabel{2015.04.22.def1}
A category with fiber products is a category together with, for all pairs of morphisms of the form $f:X\sr Z$, $g:Y\sr Z$, fiber squares
$$
\begin{CD}
(X,f)\times_Z (Y,g) @>pr^{(X,f),(Y,g)}_2>> Y\\
@Vpr^{(X,f),(Y,g)}_1 VV @VVg V\\
X @>f>> Z
\end{CD}
$$
We will often abbreviate these main notations in various ways. The morphism $pr_2\circ g=pr_1\circ f$ from $(X,f)\times(Y,g)$ to $Z$ is denoted by $f\dd g$.
\end{definition}
Given a category with fiber products, morphisms $g_i:Y_i\sr Z$, $i=1,2$ and morphisms $a:X_1\sr Y_1$, $b:X_2\sr Y_2$ denote by 
$$(a\times b)^{g_1,g_2}:((X_1,a\circ g_1)\times_Z (X_2,b\circ g_2), (a\circ g_1)\dd (b\circ g_2))\sr ((Y_1,g_1)\times_Z (Y_2,g_2),g_1\dd g_2)$$
the unique morphism over $Z$ such that
$$(a\times b)^{g_1,g_2}\circ pr_1=pr_1\circ a$$
and
$$(a\times b)^{g_1,g_2}\circ pr_2=pr_2\circ b$$
To show that $(a\times b)^{g_1,g_2}$ exists we need to check that 
$$pr_1\circ a\circ g_1=pr_2\circ b\circ g_2$$
which is immediate from the definition of the fiber product.
\begin{lemma}
\llabel{2015.05.14.l1}
In the setting introduced above suppose that we have in addition $a':X_1'\sr X_1$ and $b':X_2'\sr X_2$. Then one has
$$((a'\circ a)\times(b'\circ b))^{g_1,g_2}=(a'\times b')^{a\circ g_1, b\circ g_2}\circ (a\times b)^{g_1,g_2}$$
\end{lemma}
\begin{proof}
Straightforward rewriting to compute the compositions of both sides with $pr^{g_1,g_2}_1$ and $pr^{g_1,g_2}_2$.
\end{proof}

\begin{definition}
\llabel{2015.03.27.def1}
A locally cartesian closed structure on a (pre-)category $\cal C$ is a collection of data of the form:
\begin{enumerate}
\item A structure of a category with fiber products on $\cal C$.
\item For all $f$, $g$ of the form $f:X\sr Z$, $g:Y\sr Z$, an object $\uu{Hom}_Z((X,f),(Y,g))$ and a morphism 
$$f \triangle g:\uu{Hom}_Z((X,f),(Y,g))\sr Z$$
together with morphisms of the form 
$$\uu{Hom}((X,f),a):\uu{Hom}((X,f),(Y,g))\sr \uu{Hom}((X,f),(Y',g'))$$
for all $a:(Y,g)\sr (Y',g')$ over $Z$, that make $\uu{Hom}((X,f),-)$ into a functor from ${\cal C}/Z$ to $\cal C$.
\item For all $f$, $g$ as above a morphism 
$$ev^{(X,f)}_{(Y,g)}:(\uu{Hom}_Z((X,f),(Y,g)), f\triangle g)\times (X,f) \sr (Y,g)$$
over $Z$ such that for all $h:W\sr Z$ the map
$$adj^{(W,h),(X,f)}_{(Y,g)}:Hom_Z((W,h),(\uu{Hom}_Z((X,f),(Y,g)),f\triangle g))\sr $$$$Hom_Z(((W,h)\times (X,f), h\dd f), (Y,g))$$
given by 
$$u\mapsto (u\times Id_{X})^{f\triangle g,f}\circ ev^{(X,f)}_{(Y,g)}$$
is a bijection and such that the morphisms $ev^{(X,f)}_{(Y,g)}$ are natural in $Y$.
\end{enumerate}
A locally cartesian closed (pre-)category is a (pre-)category together with a locally cartesian closed structure on it.
\end{definition}
If a locally cartesian closed category is given with a final object $pt$ we will write $X\times Y$ for $(X,\pi_X)\times_{pt}(Y,{\pi_Y})$ where $\pi_X$ and $\pi_Y$ are the unique morphisms from $X$ and $Y$ respectively to $pt$.

By definition the objects $(\uu{Hom}((X,f),(Y,g)),f\triangle g)$ of ${\cal C}/Z$ are functorial only in $(Y,g)$. Their functoriality in $(X,f)$ is a consequence of a lemma. For $f:X\sr Z$, $f':X'\sr Z$, $g:Y\sr Z$ and $h:X'\sr X$ such that $h\circ f=f'$ let 
$$\uu{Hom}_Z(h,(Y,g)):\uu{Hom}_Z((X,f),(Y,g))\sr \uu{Hom}_Z((X',f'),(Y,g))$$
be the unique map whose adjoint 
$$adj(\uu{Hom}_Z(h,(Y,g))):(\uu{Hom}_Z((X,f),(Y,g)), f\triangle g)\times_Z (X',f')\sr (Y,g)$$
equals $(Id_{\uu{Hom}_Z((X,f),(Y,g))}\times h)^{f\triangle g, f}\circ ev^{X}_{Y}$. Then one has:
\begin{lemma}
\llabel{2015.04.10.l1}
The morphisms $\uu{Hom}_Z(h,(Y,g))$ satisfy the equations
$$\uu{Hom}_Z(h,(Y,g))\circ (f'\triangle g)=f\triangle g$$
and the equations
$$\uu{Hom}_Z(h_1\circ h_2,(Y,g))=\uu{Hom}(h_2,(Y,g))\circ \uu{Hom}(h_1,(Y,g))$$
$$\uu{Hom}_Z(Id,(Y,g))=Id$$
making $\uu{Hom}_Z(-,(Y,g))$ into a contravariant functor from ${\cal C}/Z$ to itself. In addition, for each $h':(Y,g)\sr (Y,g')$ the square
$$
\begin{CD}
\uu{Hom}_Z((X',f'),(Y,g)) @>\uu{Hom}_Z((X',f'),h')>> \uu{Hom}_Z((X',f'),(Y',g'))\\
@V\uu{Hom}_Z(h,(Y,g)) VV @VV\uu{Hom}_Z(h,(Y',g')) V\\
\uu{Hom}_Z((X,f),(Y,g)) @>\uu{Hom}_Z((X,f),h')>> \uu{Hom}_Z((X,f),(Y',g'))
\end{CD}
$$
commutes.
\end{lemma}
\begin{proof}
It is a particular case of \cite[Theorem 3, p.100]{MacLane}. The commutativity of the square is a part of the "bifunctor" claim of the theorem. 
\end{proof}
\begin{lemma}
\llabel{2015.04.20.l2}
In a locally cartesian closed category let $f:X\sr Z$, $f':X'\sr Z$, $g:Y\sr Z$ be objects over $Z$ and let $a:X'\sr X$ be a morphism over $Z$. Then the square
$$
\begin{CD}
(\uu{Hom}((X,f),(Y,g)),f\triangle g)\times_Z(X',f')  @>1>>  (\uu{Hom}((X,f),(Y,g)), f\triangle g)\times_Z (X,f) \\
@V2VV @VVev V\\
(\uu{Hom}_Z((X',f'),(Y,g)),f'\triangle g)\times_Z(X',f')  @>ev'>> Y
\end{CD}
$$
where $1$ is $(Id_{\uu{Hom}((X,f),(Y,g))}\times a)^{f\triangle g, f}$ and $2$ is $(\uu{Hom}(a, (Y,g))\times Id_{X'})^{f'\triangle g, f'}$, commutes.
\end{lemma}
\begin{proof}
Let us show that both paths in the square are adjoints to $\uu{Hom}(a, (Y,g))$. For the path that goes through the upper right corner it follows from the definition of  $\uu{Hom}(a, (Y,g))$ as the morphism whose adjoint is $(Id\times a)\circ ev$. For the path that goes through the lower left corner it follows from the definition of adjoint applied to $\uu{Hom}(a,(Y,g))$. Indeed, the adjoint to this morphism is
$$adj(\uu{Hom}(a,(Y,g)))=(\uu{Hom}(a,(Y,g))\times Id_{X'})\circ ev'$$
\end{proof}

\begin{lemma}
\llabel{2015.05.12.l2}
Let $\cal C$ be a locally cartesian closed category. Let $Z$, $(X,f),(Y,g),(W,h)$ be as above. 
\begin{enumerate}
\item Let $(Y',g')$ be an object over $Z$ and $a:(Y,g)\sr (Y',g')$ a morphism over $Z$. Then for any $b\in Hom_Z((W,h),\uu{Hom}_U((X,f),(Y,g)))$ one has
$$adj(b\circ \uu{Hom}_Z((X,f),a))=adj(b)\circ a$$
\item Let $(X',f')$ be an object over $Z$ and $a:(X',f')\sr (X,f)$ a morphism over $Z$. Then for any $b\in Hom_Z((W,h),\uu{Hom}_U((X,f),(Y,g)))$  one has
$$adj(b\circ \uu{Hom}_Z(a,(Y,g)))=(Id_{W}\times a)^{h,f}\circ adj(b)$$
\item Let $(W',h')$ be an object over $Z$ and $a:(W',h')\sr (W,h)$ a morphism over $Z$. Then for any $b\in Hom_Z((W,h),\uu{Hom}_U((X,f),(Y,g)))$ one has
$$adj(a\circ b)=(a\times Id_{X})^{h,f}\circ adj(b)$$
\end{enumerate}
\end{lemma}
\begin{proof}
The proof of the first case is given by 
$$adj(b\circ \uu{Hom}_Z((X,f),a))=((b\circ \uu{Hom}_Z((X,f),a))\times Id_X)^{f\triangle g', f}\circ ev^{(X,f)}_{(Y',g')}=$$
$$(b\times Id_X)^{f\triangle g, f}\circ (\uu{Hom}_Z((X,f),a))\times Id_X)^{f\triangle g', f}\circ ev^{(X,f)}_{(Y',g')}=$$$$(b\times Id_X)^{f\triangle g, f}\circ ev^{(X,f)}_{(Y,g)}\circ a=adj(b)\circ a$$
where the second equality holds by Lemma \ref{2015.05.14.l1} and the third equality by the naturality axiom for morphisms $ev^{(X,f)}_{(Y,g)}$ in $(Y,g)$. 

The proof of the second case is given by the following sequence of equalities where we use the notation $Hm$ for $\uu{Hom}_Z(a,(Y,g))$ as well as a number of other abbreviations:
$$adj(b\circ Hm)=((b\circ Hm)\times Id)\circ ev =(b\times Id)\circ (Hm\times Id)\circ ev=(b\times Id)\circ adj(Hm)=$$
$$(b\times Id)\circ (Id\times a)\circ ev=(b\times a)\circ ev=(Id\times a)\circ (b\times Id)\circ ev=(Id\times a)\circ adj(b)$$
The proof of the third case is given by
$$adj(a\circ b)=((a\circ b)\times Id_X) \circ ev^{(X,f)}_{(Y,g)}=(a\times Id_X)\circ (b\times Id_X)\circ ev^{(X,f)}_{(Y,g)}=$$
$$(a\times Id_X)\circ adj(b)$$
where the second equality holds by Lemma \ref{2015.05.14.l1}. 

Lemma is proved. 
\end{proof}

\begin{example}\rm
\llabel{2015.05.20.ex1}
The following example shows that there can be many different structures of a category with fiber products on a (pre-)category and also many locally cartesian closed structures.

Let us take as our (pre-)category the (pre-)category $preStn$ whose objects are natural numbers and $Hom(n,m)=Hom(\{1,\dots,n\},\{1,\dots,m\})$.  

Since every isomorphism class contains exactly one object every auto-equivalence of this category is an automorphism. Let $F$ be such an automorphisms. It is easy to see that it must be identity on the set of objects. Let $X=\{1,2\}$. Consider $F$ on $End(X)$. Since $F$ must respect unity and compositions, $F$ must take $Aut(X)$ to itself and must act on it by identity. If $1$ and $\sigma$ are the two elements of $Aut(X)$ we conclude that $F(1)=1$ and $F(\sigma)=\sigma$. 

Let us choose now any structure $str_0$ of a category with fiber products on $preStn$ and let us consider two structures $str_1$ and $str_{\sigma}$ that are obtained by choosing all the fiber squares as in $str_0$ and the square for the pair $(Id_X,Id_X)$ to be, correspondingly, as follows:
\begin{eq}\llabel{2015.05.20.sq1}
\begin{CD}
X @>Id_X>> X\\
@V Id_X VV @VV Id_X V\\
X @>Id_X>> X
\end{CD}
\spc\spc\mbox{\rm for $str_1$ and}\spc\spc
\begin{CD}
X @>\sigma>> X\\
@V \sigma VV @VV Id_X V\\
X @>Id_X>> X
\end{CD}
\spc\spc\mbox{\rm for $str_{\sigma}.$}
\end{eq}
The preceding discussion of the auto-equivalences of $preStn$ shows that there is no auto-equivalence which would transform $str_1$ into $str_{\sigma}$. 

The (pre-)category $preStn$ also has a locally cartesian closed structure that can be modified so that its underlying fiber product structures are $str_1$ and $str_{\sigma}$. This shows that $preStn$ has at least two locally cartesian closed structures that are not interchanged by auto-equivalences of $preStn$.
\end{example}
\begin{remark}\rm
\llabel{2015.05.20.rem1}
The previous example has a continuation in the univalent foundations where there is a notion of a category and pre-category. There one expects it to be true that the type of fiber square structures and the type of locally cartesian closed structures on a category (as opposed to those on a general pre-category) are of h-level 1, i.e., classically speaking are either empty or contain only one element. 

In addition any such structure on a pre-category should define a structure of the same kind on the Rezk completion of this pre-category with all the different structures on the pre-category becoming equal on the Rezk completion. In the case of the previous example the Rezk completion of $preStn$ is the category $FSets$ of finite sets and in view of the univalence axiom for finite sets the two pull-back squares of \ref{2015.05.20.sq1} will become equal in $FSets$. 
\end{remark}

\def\cprime{$'$}


\begin{thebibliography}{1}

\bibitem{RezkCompletion}
Benedikt Ahrens, Chris Kapulkin, and Michael Shulman.
\newblock Univalent categories and {R}ezk completion.
\newblock {\em \url{http://arxiv.org/abs/1303.0584}}, 2011.

\bibitem{BCH}
Marc Bezem, Thierry Coquand, and Simon Huber.
\newblock A model of type theory in cubical sets.
\newblock {\em \url{http://www.cse.chalmers.se/~coquand/mod1.pdf}}, 2014.

\bibitem{Cartmell0}
John Cartmell.
\newblock Generalised algebraic theories and contextual categories.
\newblock {\em Ph.D. Thesis, Oxford University}, 1978.
\newblock \url{https://uf-ias-2012.wikispaces.com/Semantics+of+type+theory}.

\bibitem{Cartmell1}
John Cartmell.
\newblock Generalised algebraic theories and contextual categories.
\newblock {\em Ann. Pure Appl. Logic}, 32(3):209--243, 1986.

\bibitem{MacLane}
S.~MacLane.
\newblock {\em Categories for the working mathematician}, volume~5 of {\em
  Graduate texts in {M}athematics}.
\newblock Springer-Verlag, 1971.

\bibitem{Streicher}
Thomas Streicher.
\newblock {\em Semantics of type theory}.
\newblock Progress in Theoretical Computer Science. Birkh\"auser Boston Inc.,
  Boston, MA, 1991.
\newblock Correctness, completeness and independence results, With a foreword
  by Martin Wirsing.

\bibitem{CMUtalk}
Vladimir Voevodsky.
\newblock The equivalence axiom and univalent models of type theory.
\newblock {\em arXiv 1402.5556}, pages 1--11, 2010.

\bibitem{Cfromauniverse}
Vladimir Voevodsky.
\newblock A {C-system} defined by a universe in a category.
\newblock {\em arXiv 1409.7925, submitted}, pages 1--7, 2014.

\bibitem{Csubsystems}
Vladimir Voevodsky.
\newblock Subsystems and regular quotients of {C-systems}.
\newblock In {\em Conference on Mathematics and its Applications, (Kuwait City,
  2014)}, number to appear, pages 1--11, 2015.

\end{thebibliography}

\end{document}